\documentclass{amsart}
\usepackage{amscd}
\usepackage{mathrsfs}
\usepackage{cases}
\usepackage{latexsym}
\usepackage{amsmath}
\usepackage{indentfirst}
\usepackage[arrow,matrix]{xy}
\usepackage{stmaryrd}
\usepackage{amsfonts}
\usepackage{amsmath,amssymb,amscd,bbm,amsthm,mathrsfs,dsfont}
\usepackage{graphicx}
\usepackage{fancyhdr}
\usepackage{amsxtra,ifthen}
\usepackage{verbatim}
\usepackage{latexsym,epsfig,amsmath}
\usepackage{epsfig}

\DeclareMathOperator{\End}{End} 
 \DeclareMathOperator{\Ker}{Ker}
\DeclareMathOperator{\Span}{Span} \DeclareMathOperator{\sgn}{sgn}

 \DeclareMathOperator{\Ann}{Ann}

\numberwithin{equation}{section}

\theoremstyle{plain}
\newtheorem{theorem}{Theorem}[section]
\newtheorem{lemma}[theorem]{Lemma}
\newtheorem{proposition}[theorem]{Proposition}

\theoremstyle{definition}
\newtheorem{definition}[theorem]{Definition}

\begin{document}

\title[On Tensor Spaces for Rook Monoid Algebras]{On Tensor Spaces for Rook Monoid Algebras}
\author{Zhankui Xiao}

\address{Xiao: School of Mathematical Sciences, Huaqiao University,
Quanzhou, Fujian, 362021, P. R. China}

\email{zhkxiao@hqu.edu.cn}

\thanks{The work of the author is supported by the National Natural Science Foundation of China
(Grant No. 11301195) and a research foundation of Huaqiao University (Project 2014KJTD14).}


\begin{abstract}
Let $m,n\in \mathbb{N}$, and $V$ be a $m$-dimensional vector space
over a field $F$ of characteristic $0$. Let $U=F\oplus V$ and $R_n$
be the rook monoid. In this paper, we construct a certain quasi-idempotent
in the annihilator of $U^{\otimes n}$ in $FR_n$, which comes from
some one-dimensional two-sided ideal of rook monoid algebra. We show
that the two-sided ideal generated by this element is indeed the whole
annihilator of $U^{\otimes n}$ in $FR_n$.
\end{abstract}

\maketitle

\section{Introduction}\label{xxsec1}

Let $m,n\in \mathbb{N}$ and $V$ be a $m$-dimensional vector space
over a field $F$ of characteristic $0$. Let $U=F\oplus V$ and ${\rm GL}(V)$
be the general linear group over $V$. We further consider $F$ as the trivial ${\rm GL}(V)$-module.
This allows us to consider $U$ and hence $U^{\otimes n}$ as a ${\rm GL}(V)$-module.
Let $R_n$ be the rook monoid, see Section \ref{xxsec2} for precise definition.
By \cite{Solomon}, there is a left action of $R_n$ on $U^{\otimes n}$ which commutes
with the left action of ${\rm GL}(V)$. Let $\varphi, \psi$ be the natural algebra
homomorphisms:
$$\begin{aligned} \varphi:&\,\,\,
FR_n\rightarrow
\End_{{\rm GL}(V)}\bigl(U^{\otimes n}\bigr),\\
\psi:&\,\,\,
F{\rm GL}(V)\rightarrow\End_{FR_n}\bigl(U^{\otimes n}\bigr),
\end{aligned}$$
respectively. The following results are proved by Solomon \cite[Theorem 5.10 and Corollary 5.18]{Solomon}.

\begin{theorem} \label{xx1.1}
1) Both $\varphi$ and $\psi$ are surjective;

2) if $m\geq n$, then $\varphi$ is an isomorphism.
\end{theorem}

The above theorem is an analogue, for $R_n$ and ${\rm GL}(V)$, of the
Schur-Weyl duality for symmetric group $\mathfrak{S}_n$ and general linear group ${\rm GL}(V)$.
When $m<n$, the algebra homomorphism $\varphi$ is in general not injective.
Therefore it is natural to ask how to describe the kernel of the
homomorphism $\varphi$, i.e., the annihilator of $U^{\otimes n}$ in the algebra
$FR_n$. This question is closely related to the invariant theory, see \cite{GW,Weyl}.

Let $G$ be an algebraic subgroup of ${\rm GL}(V)$ and $M$ be a $G$-module. One formulation
of the invariant theory for $G$ is to describe the endomorphism algebra $\End_G(M^{\otimes n})$.
It should be remarked that, in the classical invariant theory, $G$-module $M$ is usually set
as the natural representation $V$. The first fundamental theorem of invariant theory
provides generators of $\End_G(M^{\otimes n})$ and the second fundamental theorem of invariant theory
describes all the relations among the generators. From this point of view, the above
Theorem \ref{xx1.1} can be seen as the first fundamental theorem of invariant theory
for general linear group ${\rm GL}(V)$ about the module $U$. Therefore, it is desirable
to give out the second fundamental theorem, i.e., a characterization of the annihilator ideal
of $U^{\otimes n}$ in rook monoid algebra $FR_n$ by its standard generators.

The purpose of this article is to answer the above question. Recently Hu and the author
\cite{HX} proved the second fundamental theorem for symplectic group and Lehrer-Zhang in \cite{LZ}
gave out the second fundamental theorem for orthogonal group, where they deeply used
the different versions of invariant theory. In both symplectic and orthogonal cases,
the annihilator of $n$-tensor space in the specialized Brauer algebra is generated by an
explicitly described quasi-idempotent. Motivated by the articles \cite{HX,LZ}, we construct
a certain quasi-idempotent $Y_{m+1}$ (see Section \ref{xxsec4}) in $\Ker \varphi$ and prove that

\begin{theorem} \label{xx1.2}
With notations as above, if $m<n$, we have
$$
\Ann_{FR_n}\bigl( U^{\otimes n}\bigr)=\langle Y_{m+1} \rangle.
$$
\end{theorem}

We would like to point out that an analogue of Theorem \ref{xx1.1} for orthogonal group
and rook Brauer algebra (also called partial Brauer algebra) was obtained by
Halverson-delMas \cite{Hd} and Martin-Mazorchuk \cite{MaM} independently. We conjecture
that there exists an analogue of Theorem \ref{xx1.2} for rook Brauer algebra and
will consider this question in a future separate article.

The content of this article is organized as follows. In Section \ref{xxsec2} we recall some basic
knowledge about the structure and representation theory of rook monoids as well as
some combinatorics which are needed later. In Section \ref{xxsec3} we construct the one-dimensional
two-sided ideals in the rook monoid algebra $FR_n$. Furthermore, we can get the block
decomposition of $FR_n$ by the theory of Specht modules. In Section \ref{xxsec4} we prove our main result
Theorem \ref{xx1.2} and the proof will be proceed in three steps.

\section{Preliminaries}\label{xxsec2}

\subsection{Rook monoid}\label{xxsec2.1}

Let $R_n$ be the set of all $n\times n$ matrices that contain at most one entry equal to $1$
in each row and column and zeros elsewhere. With the operation of matrix multiplication,
$R_n$ has the structure of a monoid. The monoid $R_n$ is known both as the {\em rook monoid}
and the {\em symmetric inverse semigroup} \cite{Solomon0}. The number of rank $r$ matrices in $R_n$ is
${n \choose r}^2 r!$ and hence the rook monoid has a total of $\sum_{r=0}^n {n \choose r}^2 r!$
elements.

A presentation of the rook monoid $R_n$ is given in \cite{KM} which is more helpful for us
(see also \cite[Section 2]{Hd}). The rook monoid $R_n$ is generated by
$s_1,\cdots,s_{n-1},p_1,\cdots,p_n$ subject to the following relations:
$$
\begin{aligned}
& s_i^2=1, & 1\leqslant i\leqslant n-1,\\
& s_is_j=s_js_i, & |i-j|>1,\\
& s_is_{i+1}s_i=s_{i+1}s_is_{i+1}, & 1\leqslant i\leqslant n-2,\\
& p_i^2=p_i, & 1\leqslant i\leqslant n,\\
& p_ip_j=p_jp_i, & i\neq j,\\
& s_ip_i=p_{i+1}s_i, & 1\leqslant i\leqslant n-1,\\
& s_ip_j=p_js_i, & |i-j|>1,\\
& p_is_ip_i=p_ip_{i+1}, & 1\leqslant i\leqslant n-1.\\
\end{aligned}
$$
It is clear that the symmetric group $\mathfrak{S}_n \subseteq R_n$.

Now we recall another presentation of $R_n$ by rook $n$-diagram (see \cite{Hd,KM}).
A rook $n$-diagram is a graph consisting of two rows each with $n$ vertices
such that each vertex in the top row is connected to at most one vertex in the
bottom row. We denote ${\rm Rd}_n$ the set of all rook $n$-diagram. For each rook $n$-diagram $D$,
we shall label the vertices in the top row of $D$ by $1,2,\ldots,n$ from left to right,
and label the vertices in the bottom row of $D$ by $1^-,2^-,\ldots,n^-$ also from left to right.
For example, let $D$ be the following rook diagram
%\begin{figure}
\begin{center}
\epsfig{figure=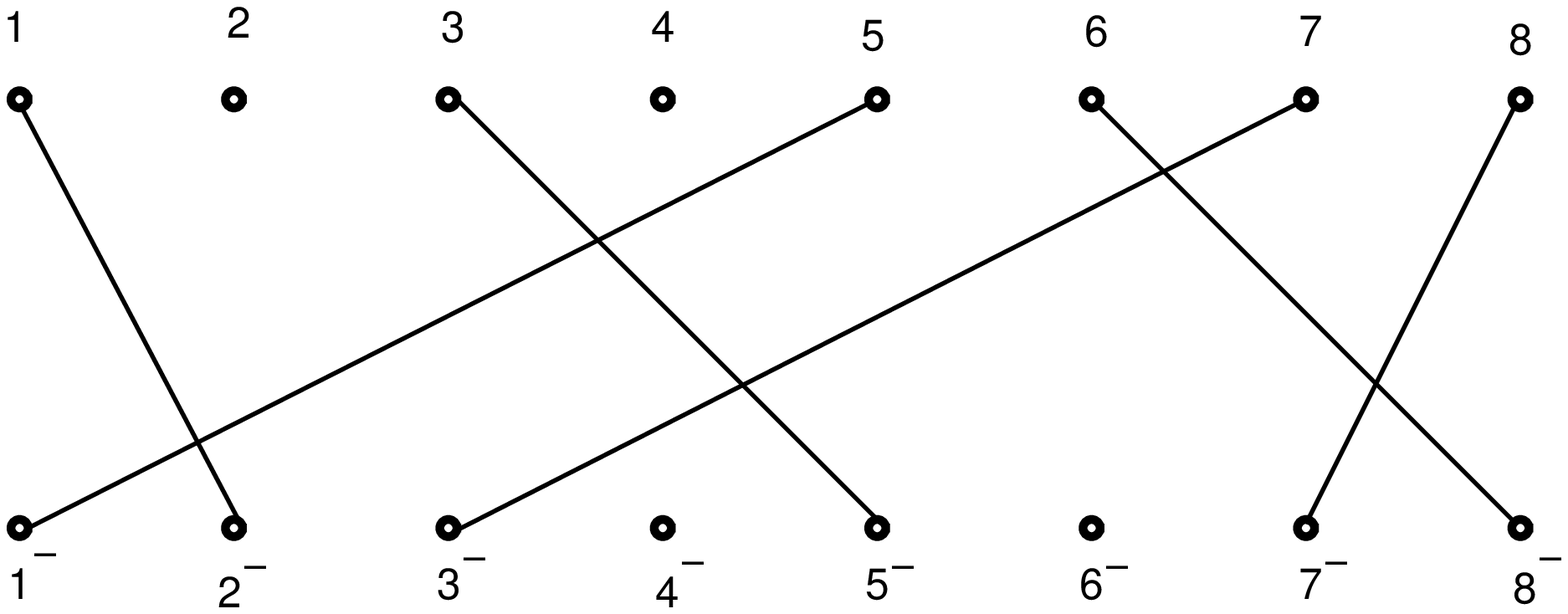,clip,height=3.0cm}
\end{center}
%\end{figure}
Then $D\in {\rm Rd}_8$. For a rook $n$-diagram, the vertices which are not incident to an edge
are called {\em isolated vertices}.
The multiplication of two rook $n$-diagram is defined using natural
concatenation of diagrams. Precisely, we compose two rook $n$-diagrams $D_1, D_2$ by identifying the
bottom row of vertices in $D_1$ with the top row of vertices in $D_2$. The result is also a
rook $n$-diagram and defined as the multiplication $D_1\cdot D_2$.

There is a connection between the above two presentations of rook monoids. For each integer
$1\leqslant i<n$, the standard generator $s_i$ corresponds to the rook $n$-diagram with edges connecting
the vertices $i$ (resp. $i+1$) on the top row with $(i+1)^-$ (resp. $i^-$) on the bottom row,
and all other edges being vertical, connecting the vertices $k$ and $k^-$ for all $k\neq i,i+1$.
For each integer $1\leqslant j\leqslant n$, the standard generator $p_j$ corresponds to the rook $n$-diagram
with isolated vertices $j$ and $j^-$, and vertical edges $\{k,k^-\}$ for all $k\neq j$.
For an integer $r$ with $0\leqslant r\leqslant n$, we define
$$
\mathcal{D}_r:=\{d\in\mathfrak{S}_n\ |\ (1)d<(2)d<\cdots<(r)d, (r+1)d<\cdots<(n)d\}.
$$
Note that $\mathcal{D}_0=\{1\}$ and $\mathcal{D}_r$ is the set of distinguished
right coset representatives of $\mathfrak{S}_{(r,n-r)}$ in $\mathfrak{S}_n$. It is helpful to point out
that we consider the elements of symmetric group as {\em right} permutations. From this point of view,
the composition of permutations coincides with the multiplication of rook diagrams.
Let $F$ be a field of characteristic $0$. The following proposition follows directly from
the diagrammatic multiplication of rook monoid algebras.

\begin{proposition}\label{xx2.1}
For each $D\in {\rm Rd}_n$ there exists a unique quadruple $(d_1,d_2,r,\sigma)$ with
$0\leqslant r\leqslant n$, $d_1,d_2\in\mathcal{D}_r$, $\sigma\in\mathfrak{S}_{\{r+1,r+2,\ldots,n\}}$
and such that $D=d_1^{-1}p_1p_2\cdots p_r\sigma d_2$. In particular, the elements in the following set
$$
\{d_1^{-1}p_1p_2\cdots p_r\sigma d_2\ |\ 0\leqslant r\leqslant n,\ d_1,d_2\in\mathcal{D}_r,
\ \sigma\in\mathfrak{S}_{\{r+1,r+2,\ldots,n\}}\}
$$
form a basis of the rook monoid algebra $FR_n$ and it coincides with the natural basis given by
rook $n$-diagram.
\end{proposition}

Note that in the above proposition, the element $d_1^{-1}p_1p_2\cdots p_r\sigma d_2$ corresponds to
the rook $n$-diagram with the isolated vertices $(1)d_1,(2)d_1,\ldots,(r)d_1$ in the top row,
the isolated vertices $((1)d_2)^-,((2)d_2)^-,\ldots,((r)d_2)^-$ in the bottom row, and edges connecting
$(j)d_1$ with $((j)\sigma d_2)^-$ for $j=r+1,r+2,\ldots,n$.

As is predicated in the introduction, there is a left action of $FR_n$ on the $n$-tensor space
$U^{\otimes n}$ which commutes with the left action of ${\rm GL}(V)$. We now recall the definition
of this action. Let $\delta_{i,j}$ denote the value of the usual Kronecker delta. We fix a basis
$\{v_1,v_2,\ldots,v_m\}$ of $V$ and a basis $\{v_0\}$ of $F$ such that
$$
U^{\otimes n}=F-\Span\{v_{i_1}\otimes v_{i_2}\otimes\cdots\otimes v_{i_n}
\ |\ i_j=0,1,\ldots,m\}.
$$
The left action of $FR_n$ on $U^{\otimes n}$ is defined on generators by (see \cite{Solomon})
$$\begin{aligned}
s_j(v_{i_1}\otimes\cdots\otimes v_{i_n}) &:=v_{i_1}\otimes\cdots\otimes v_{i_{j-1}}\otimes
v_{i_{j+1}}\otimes v_{i_j}\otimes v_{i_{j+2}}\otimes \cdots\otimes v_{i_n},\\
p_j(v_{i_1}\otimes\cdots\otimes v_{i_n}) &:=\delta_{i_{j},0} v_{i_1}\otimes\cdots\otimes v_{i_{j-1}}\otimes
v_0\otimes v_{i_{j+1}}\otimes\cdots\otimes v_{i_n}.
\end{aligned}$$
\vspace{2pt}

\subsection{Specht module}\label{xxsec2.2}

Munn demonstrated the representations of rook monoids stemmed from his work on the general theory
of representations of finite semigroups \cite{Munn1,Munn2}. He computed the irreducible characters
of $R_n$ using irreducible characters of symmetric group $\mathfrak{S}_r$ ($1\leqslant r\leqslant n$) and showed that the
rook monoid algebra $FR_n$ is semisimple. Motivated by Munn's work and the theory of Specht modules
of symmetric groups, Grood \cite{Grood} studied the theory of Specht modules of $FR_n$ which we recall here. It should
be pointed out that although Grood worked on the complex field $\mathbb{C}$, it is clear that
all the results in \cite{Grood} also hold for an arbitrary field $F$ of characteristic $0$.

A partition of $r$ is a sequence of nonnegative integers $\lambda=(\lambda_1,\lambda_2,\cdots)$ with
$\lambda_1\geqslant \lambda_2\geqslant \cdots$ and $\sum_{i\geqslant 1}\lambda_i=r$. In this case,
we write $\lambda\vdash r$ and $|\lambda|=r$. The length of $\lambda$, denoted $\ell(\lambda)$, is the
maximum subscript $j$ such that $\lambda_j>0$. The Young diagram of $\lambda$ is defined to be the set
$$
[\lambda]:=\{(i,j)\ |\ 1\leqslant j\leqslant \lambda_i \}.
$$
A $\lambda$-tableau is a bijection $\mathfrak{t}:[\lambda]\rightarrow \{1,2,\ldots,r\}$.

\begin{definition}\label{xx2.2}
Let $\lambda\vdash r$ with $0\leqslant r\leqslant n$. An $n$-tableau of shape $\lambda$,
also called a $\lambda_r^n$-tableau, is a bijection $\mathfrak{t}:[\lambda]\rightarrow S$,
where $S$ is a subset of $r$ distinct elements of the set $\{1,2,\ldots,n\}$.
\end{definition}

Let $\mathfrak{t}$ be a $\lambda_r^n$-tableau. We write $\mathfrak{t}_{ij}=\mathfrak{t}(i,j)$,
the entry contained in the box of the $i$-th row and $j$-th column of $\mathfrak{t}$ by Grood's notations.
The {\em content} of $\mathfrak{t}$, denoted ${\rm cont}(\mathfrak{t})$, is the image of $\mathfrak{t}$.
For a given $\lambda_r^n$-tableau $\mathfrak{t}$, let $\mathcal{R}_i$ (resp. $\mathcal{C}_j$) be the
set of entries in the $i$-th row (resp. the $j$-th column) of $\mathfrak{t}$. Two $\lambda_r^n$-tableaux
$\mathfrak{t}$ and $\mathfrak{s}$ are called {\em row-equivalent} if the corresponding rows of
the two tableaux contain the same entries. In other words, the set $\mathcal{R}_i$ of $\mathfrak{t}$ coincides with
that of $\mathfrak{s}$ for $1\leqslant i\leqslant \ell(\lambda)$. In this case, we write
$\mathfrak{t}\sim \mathfrak{s}$.

\begin{definition}\label{xx2.3}
An $n$-tabloid $\{\mathfrak{t}\}$ of shape $\lambda$, also called a $\lambda_r^n$-tabloid $\{\mathfrak{t}\}$,
is the set of all $\lambda_r^n$-tableaux that are row-equivalent to $\mathfrak{t}$; i.e.,
$$
\{\mathfrak{t}\}=\{\mathfrak{s}\ |\ \mathfrak{s}\sim \mathfrak{t}\}.
$$
\end{definition}

Let $\lambda\vdash r$ with $0\leqslant r\leqslant n$, $N^{\lambda}$ the $F$-vector space generated by all
$\lambda_r^n$-tableau. We can define an action of $R_n$ on $N^{\lambda}$ by first determining how $R_n$
acts on the basis of $\lambda_r^n$-tableau and then linearly extending this action to the whole space.
If $\mathfrak{t}$ is a $\lambda_r^n$-tableau, for each $\pi\in R_n$ we define $\pi\mathfrak{t}$ to be the
zero vector if there exists $\mathfrak{t}_{ij}\in {\rm cont}(\mathfrak{t})$ such that $\pi(\mathfrak{t}_{ij})=0$;
otherwise, we say that $(\pi\mathfrak{t})_{ij}=\pi(\mathfrak{t}_{ij})$. Let us rephrase this action by the
language of rook $n$-diagram. Let $D=d_1^{-1}p_1p_2\cdots p_s\sigma d_2$ be a rook $n$-diagram as that in
Proposition \ref{xx2.1}. Then $D\mathfrak{t}=0$ if there exists an integer $i$ with $1\leqslant i\leqslant s$
such that $(i)d_2\in {\rm cont}(\mathfrak{t})$. Otherwise (in this case ${\rm cont}(\mathfrak{t})\subseteq \{(s+1)d_2,\cdots,(n)d_2\}$),
$D$ maps the entry which equal to $(j)d_2$ to $(j)\sigma^{-1}d_1$ for some $s+1\leqslant j\leqslant n$.

Let $M^{\lambda}$ be the $F$-vector space generated by all $\lambda_r^n$-tabloids. If two $\lambda_r^n$-tableaux
$\mathfrak{s}, \mathfrak{t}$ satisfies $\mathfrak{s}\sim \mathfrak{t}$, then $\pi\mathfrak{s}\sim \pi\mathfrak{t}$
for all $\pi\in R_n$. Therefore we have an induced action of $R_n$ on $M^{\lambda}$ given by
\begin{gather*}
\pi\{\mathfrak{t}\}=\left\{
\begin{tabular}
[c]{l}%
$0$\ \ \ \ \ \ \ \ if $\pi\mathfrak{t}=0$,\\
$\{\pi\mathfrak{t}\}$\ \ \ \ otherwise.%
\end{tabular}
\right.
\end{gather*}
For a $\lambda_r^n$-tableau $\mathfrak{t}$, let $C_{\mathfrak{t}}:=\mathfrak{S}_{\mathcal{C}_1}\times \mathfrak{S}_{\mathcal{C}_2}\times
\cdots\times \mathfrak{S}_{\mathcal{C}_l}$ be the column stabiliser subgroup of $\mathfrak{t}$ in
symmetric group $\mathfrak{S}_n$, where $l=\lambda_1$. Note that $C_{\mathfrak{t}}$
does not consist of all the elements in $R_n$ that fix the columns of $\mathfrak{t}$.
For each $\lambda_r^n$-tableau $\mathfrak{t}$ we define the following element in $M^{\lambda}$:
$$
e_{\mathfrak{t}}:=\sum_{\sigma\in C_{\mathfrak{t}}} \sgn(\sigma)\sigma\{\mathfrak{t}\},
$$
where $\sgn(\sigma)$ stands for the sign of the permutation $\sigma$. The element $e_{\mathfrak{t}}$
is called the {\em $n$-polytabloid} associated with $\mathfrak{t}$.

\begin{lemma}\label{xx2.4} {\rm (\cite[Proposition 3.3]{Grood})}
Suppose $\pi\in R_n$, $\mathfrak{t}$ is a $\lambda_r^n$-tableau. If $\pi\mathfrak{t}=0$,
then $\pi e_{\mathfrak{t}}=0$. Otherwise, $\pi e_{\mathfrak{t}}=e_{\widehat{\pi}\mathfrak{t}}$.
\end{lemma}

We refer the reader to \cite{Grood} for the definition of $\widehat{\pi}$ which will not be used
in our paper. For each partition $\lambda\vdash r$ with $0\leqslant r\leqslant n$, let us define
$$
R^{\lambda}:=F-\Span\{e_{\mathfrak{t}}\ |\ \mathfrak{t}\ \text{is a}\ \lambda_r^n\text{-tableau}\}.
$$
It is clear that the subspace $S^{\lambda}$ of $R^{\lambda}$ generated by those $n$-polytabloids
$e_{\mathfrak{t}}$ with ${\rm cont}(\mathfrak{t})=\{1,2,\ldots,r\}$ is exactly the so-called
Specht module for symmetric group $\mathfrak{S}_r$. Therefore Grood called $R^{\lambda}$ the
{\em Specht module} for rook monoid algebra $FR_n$. We conclude the main results of \cite{Grood}
as follows

\begin{theorem}\label{xx2.5}
With notations as above, $\{R^{\lambda}\ |\ \lambda\vdash r, 0\leqslant r\leqslant n\}$ forms
a complete set of pairwise non-isomorphic irreducible $FR_n$-modules.
\end{theorem}

\section{Blocks of rook monoid algebra}\label{xxsec3}

In this section, we will decompose the rook monoid algebra $FR_n$ as the direct sum of blocks,
the indecomposible two-sided ideals, by constructing elements in $FR_n$ analogous to the Young
symmetrizers or anti-symmetrizers of symmetric groups.

\begin{lemma}\label{xx3.1}
Let $\rho_1,\rho_2,\rho_3$ be the three one-dimensional representations of $FR_n$ which are
defined on generators by
$$
\rho_1(s_i)=1,\hspace{6pt} \rho_1(p_j)=0;
$$
$$
\rho_2(s_i)=-1,\hspace{6pt} \rho_2(p_j)=0;
$$
$$
\rho_3(s_i)=1,\hspace{6pt} \rho_3(p_j)=1,
$$
where $1\leqslant i\leqslant n-1$ and $1\leqslant j\leqslant n$. For $n\geqslant 2$, then up
to isomorphism, $\rho_1,\rho_2,\rho_3$ are the only three non-isomorphic one-dimensional representations of $FR_n$.
\end{lemma}

\begin{proof}
Using the generators and relations for $FR_n$, one checks easily that $\rho_1,\rho_2,\rho_3$ are
three one-dimensional representations of $FR_n$.

Suppose that $Fv$ affords a one-dimensional representation $\rho$ of $FR_n$. Since $s_1^2=1$ and
$p_1^2=p_1$, there are only three possibilities:

\textbf{Case 1.} $p_1v=0$ and $s_1v=v$. Using the relations $s_is_{i+1}s_i=s_{i+1}s_is_{i+1}$
and $s_ip_i=p_{i+1}s_i$, we deduce that $s_iv=v$ and $p_jv=0$ for all $1\leqslant i\leqslant n-1$
and $1\leqslant j\leqslant n$. Hence $\rho=\rho_1$ in this case.

\textbf{Case 2.} $p_1v=0$ and $s_1v=-v$. Using the relations $s_is_{i+1}s_i=s_{i+1}s_is_{i+1}$
and $s_ip_i=p_{i+1}s_i$, we deduce that $s_iv=-v$ and $p_jv=0$ for all $1\leqslant i\leqslant n-1$
and $1\leqslant j\leqslant n$. Hence $\rho=\rho_2$ in this case.

\textbf{Case 3.} $p_1v=v$. Using the relation $s_ip_i=p_{i+1}s_i$ we get that $p_jv=v$ for all
$1\leqslant j\leqslant n$. Then using the relation $p_is_ip_i=p_ip_{i+1}$, we deduce that
$s_iv=v$ for all $1\leqslant i\leqslant n-1$. Hence $\rho=\rho_3$ in this case.
This completes the proof of the lemma.
\end{proof}

It is clear that the two-sided ideal $\langle p_1p_2\cdots p_n\rangle$ generalized by $p_1p_2\cdots p_n$ corresponds to the
one-dimensional representation $\rho_3$. Since $FR_n$ is semisimple with base field $F$
of characteristic $0$ \cite{Munn2}, there are two one-dimensional two-sided ideals
corresponding to the two non-isomorphic one-dimensional representations $\rho_1,\rho_2$.
Our first task is to construct these two one-dimensional two-sided ideals in an explicit way.
Let $w\in \mathfrak{S}_n$. We denote $\ell(w)$ the length function of $w$, i.e., the length
of any reduced expression of $w$.

\begin{definition}\label{xx3.2}
Let $D\in {\rm Rd}_n$ and $(d_1,d_2,r,\sigma)$ the unique quadruple with
$0\leqslant r\leqslant n$, $d_1,d_2\in\mathcal{D}_r$, $\sigma\in\mathfrak{S}_{\{r+1,r+2,\ldots,n\}}$
and such that $D=d_1^{-1}p_1p_2\cdots p_r\sigma d_2$. We define
$\ell(D):=\ell(d_1)+\ell(\sigma)+\ell(d_2)$ and $\sgn(D):=(-1)^{r+\ell(D)}$.
\end{definition}

When $r=0$, the definition $\ell(D)$ coincides with the length function on $\mathfrak{S}_n$.
For any $D_1, D_2\in {\rm Rd}_n$, note that in general
$$
\sgn(D_1D_2)\neq \sgn(D_1)\sgn(D_2).
$$
For each integer $0\leqslant r\leqslant n$, we use ${\rm Rd}_n[r]$ to denote the set of
rook $n$-diagrams which have exactly $r$ isolated vertices in each row. Let $R_n^{(r)}$
be the two-sided ideal of $FR_n$ generated by $p_1p_2\cdots p_r$. Then there is a filtration
$$
FR_n=R_n^{(0)}\supset R_n^{(1)}\supset R_n^{(2)}\supset\cdots \supset R_n^{(n)}\supset 0
$$
of two-sided ideals of $FR_n$. We have from the multiplication of rook $n$-diagrams that
the ideal $R_n^{(r)}$ has a basis ${\rm Rd}_n[r]\cup {\rm Rd}_n[r+1]\cup\cdots \cup{\rm Rd}_n[n]$.
For $i=1,2$, let $X_i\in FR_n$ be an element such that the two-sided ideal $\langle X_i\rangle$
of $FR_n$ generated by $X_i$ corresponds to the one-dimensional representation $\rho_i$.

\begin{lemma}\label{xx3.3}
The elements $X_1$, $X_2$ can be taken of the following form:
$$
X_1=\sum_{\sigma\in\mathfrak{S}_n}\sigma +\sum_{r=1}^{n}\sum_{D\in {\rm Rd}_n[r]}C_D D,
$$
$$
X_2=\sum_{\sigma\in\mathfrak{S}_n}\sgn(\sigma)\sigma +\sum_{r=1}^{n}\sum_{D\in {\rm Rd}_n[r]}C'_D\sgn(D) D,
$$
where $C_D, C'_D\in F$.
\end{lemma}

\begin{proof}
It is clear that there is an algebra epimorphism
$$
\theta:\ FR_n\twoheadrightarrow FR_n/R_n^{(1)}\cong F\mathfrak{S}_n.
$$
Since the algebras $FR_n$ and $F\mathfrak{S}_n$ are both semisimple, it follows from Lemma \ref{xx3.1}
that the image of the two one-dimensional two-sided ideals corresponding to representations $\rho_1, \rho_2$
under $\theta$ must be the only two non-isomorphic one-dimensional two-sided ideals of $F\mathfrak{S}_n$.
The lemma now follows from the well-known results about Young symmetrizers or anti-symmetrizers of symmetric groups.
\end{proof}

For an arbitrary element $a\in FR_n$, we say that the rook $n$-diagram $D$ is involved in $a$, if
$D$ appears with nonzero coefficient when writing $a$ as a linear combination of the basis of
rook $n$-diagrams. The following lemma is well-known for symmetric groups.

\begin{lemma}\label{xx3.4}
Let $r$ be an integer with $0\leqslant r\leqslant n$. There exists a unique element $w_0\in \mathcal{D}_r$
of maximal length $r(n-r)$. If $s_{i_1}s_{i_2}\cdots s_{i_{r(n-r)}}$ is a reduced expression
of $w_0$, then for any integer $0\leqslant j\leqslant r(n-r)$, there is $s_{i_1}s_{i_2}\cdots s_{i_j}
\in\mathcal{D}_r$. Conversely, for any $d\in \mathcal{D}_r$, there exists a reduced expression
$s_{i_1}s_{i_2}\cdots s_{i_{r(n-r)}}$ of $w_0$ such that $d=s_{i_1}s_{i_2}\cdots s_{i_j}$ for some
$0\leqslant j\leqslant r(n-r)$.
\end{lemma}

\begin{lemma}\label{xx3.5}
For each integer $1\leqslant r\leqslant n$, and any $D_1, D_2\in {\rm Rd}_n[r]$, we have
$$
C_{D_1}=C_{D_2}\quad \text{and} \quad C'_{D_1}=C'_{D_2}.
$$
\end{lemma}

\begin{proof}
We only prove $C_{D_1}=C_{D_2}$ and the other assertion $C'_{D_1}=C'_{D_2}$ can be proved similarly.
If $D_1=d_1^{-1}p_1p_2\cdots p_r\sigma d_2$ and $D_2=d_3^{-1}p_1p_2\cdots p_r\sigma d_2$ with
$d_1,d_2,d_3\in\mathcal{D}_r$ and $\sigma\in\mathfrak{S}_{\{r+1,r+2,\ldots,n\}}$, we claim that
$$
C_{D_1}=C_{D_2}.
$$
In fact, by Lemma \ref{xx3.4} it suffices to prove that
$$
C_{d_1^{-1}p_1p_2\cdots p_r\sigma d_2}=C_{s_id_1^{-1}p_1p_2\cdots p_r\sigma d_2},
$$
whenever $d_1s_i\in\mathcal{D}_r$ with $\ell(d_1s_i)=\ell(d_1)+1$.
Let us compare the coefficients of $D_1=d_1^{-1}p_1p_2\cdots p_r\sigma d_2$ in both sides of
the equality $s_iX_1=X_1$. Since $s_i$ is invertible in the rook monoid $R_n$,
we have by the concatenation rule of rook diagrams that $s_i {\rm Rd}_n[f]={\rm Rd}_n[f]$
for each integer $1\leqslant f\leqslant n$. Hence it is clear that our claim holds.

Let $\ast$ be the algebra anti-automorphism of $FR_n$ which is defined on generators by
$s_i^{\ast}=s_i$, $p_j^{\ast}=p_j$ for each $1\leqslant i\leqslant n-1$ and $1\leqslant j\leqslant n$
(see \cite{East} for example). Using Lemma \ref{xx3.3}, we see that $X_1^{\ast}=X_1$ and
$X_2^{\ast}=X_2$. Now combining the fact $X_1^{\ast}=X_1$ and the above claim, we have
$$
C_{d_1^{-1}p_1p_2\cdots p_r\sigma d_2}=C_{p_1p_2\cdots p_r\sigma d_2}=C_{p_1p_2\cdots p_r\sigma}
$$
for all $d_1,d_2\in\mathcal{D}_r$ and $\sigma\in\mathfrak{S}_{\{r+1,r+2,\ldots,n\}}$.
Therefore, to prove the lemma, it suffices to show that
$$
C_{p_1p_2\cdots p_r\sigma_1}=C_{p_1p_2\cdots p_r\sigma_2}
$$
for all $\sigma_1,\sigma_2\in\mathfrak{S}_{\{r+1,r+2,\ldots,n\}}$. Equivalently, it is enough to
show that $C_{p_1p_2\cdots p_r s_i\sigma}=C_{p_1p_2\cdots p_r\sigma}$ for any
$\sigma\in\mathfrak{S}_{\{r+1,r+2,\ldots,n\}}$ and $r+1\leqslant i\leqslant n$ satisfying
$\ell(s_i\sigma)=\ell(\sigma)+1$. Let us compare the coefficients of $p_1p_2\cdots p_r\sigma$ in both sides of
the equality $s_iX_1=X_1$. Since $s_i$ is invertible in the rook monoid $R_n$,
we have $s_i {\rm Rd}_n[f]={\rm Rd}_n[f]$ for each integer $1\leqslant f\leqslant n$ and hence
$C_{p_1p_2\cdots p_r s_i\sigma}=C_{p_1p_2\cdots p_r\sigma}$. This completes the proof of the lemma.
\end{proof}

The following proposition is the first main result in this section, which reveals a
similarity between the elements $X_1, X_2$ with the symmetrizer and anti-symmetrizer in the
symmetric group case.

\begin{proposition}\label{xx3.6}
The elements $X_1$, $X_2$ can be taken of the following form:
$$
X_1=\sum_{\sigma\in\mathfrak{S}_n}\sigma +\sum_{r=1}^{n}(-1)^rr! \sum_{D\in {\rm Rd}_n[r]} D,
$$
$$
X_2=\sum_{\sigma\in\mathfrak{S}_n}\sgn(\sigma)\sigma +\sum_{D\in {\rm Rd}_n[1]}\sgn(D) D.
$$
\end{proposition}

We divide the proof of Proposition \ref{xx3.6} into two lemmas for comfortable reading.

\begin{lemma}\label{xx3.7}
The element $X_1$ can be taken of the following form:
$$
X_1=\sum_{\sigma\in\mathfrak{S}_n}\sigma +\sum_{r=1}^{n}(-1)^rr! \sum_{D\in {\rm Rd}_n[r]} D.
$$
\end{lemma}

\begin{proof}
By Lemma \ref{xx3.5}, we can take $X_1$ of the following form:
$$
X_1=\sum_{\sigma\in\mathfrak{S}_n}\sigma +\sum_{r=1}^{n}c_r \sum_{D\in {\rm Rd}_n[r]} D,
$$
where $c_r\in F$ for each $1\leqslant r\leqslant n$. It remains to compute these $c_r$ explicitly.

We first compute $c_1$. The strategy we shall use is to compare the coefficients of $p_1$ in both
sides of the equality
$$
0=p_1X_1=\sum_{\sigma\in\mathfrak{S}_n}p_1\sigma +\sum_{r=1}^{n}c_r \sum_{D\in {\rm Rd}_n[r]}p_1D. \eqno(3.1)
$$
Note that by the concatenation rule of rook diagrams $p_1D\in {\rm Rd}_n$ for all $D\in {\rm Rd}_n$
and $p_1$ is involved in $p_1D$ only if $D\in {\rm Rd}_n[0] \cup {\rm Rd}_n[1]$.
If $D\in {\rm Rd}_n[0]$, then it is easy to see that $p_1$ is involved in $p_1D$ ($p_1=p_1D$ in this case)
if and only if $D=1$, the identity element of rook monoid $R_n$. If $D\in {\rm Rd}_n[1]$, then
$p_1$ is involved in $p_1D$ (also $p_1=p_1D$ in this case) if and only if $D=p_1$. Therefore, the
coefficient of $p_1$ in the right-hand side of the equality $(3.1)$ is $1+c_1$ which implies that $c_1=-1$.

In general, suppose that $2\leqslant r\leqslant n$. The strategy we shall use to compute $c_r$ is to
compare the coefficients of $p_1p_2\cdots p_r$ in both sides of the equality $(3.1)$. We claim that
$p_1p_2\cdots p_r$ is involved in $p_1D$ (in this case, $p_1D=p_1p_2\cdots p_r$ since $p_1D\in {\rm Rd}_n$)
only if $D\in {\rm Rd}_n[r-1] \cup {\rm Rd}_n[r]$. In fact, by the concatenation rule of rook diagrams,
we know that $p_1D=p_1p_2\cdots p_r$ only if $D\in {\rm Rd}_n[0]\cup {\rm Rd}_n[1]\cup\cdots\cup {\rm Rd}_n[r]$.
However, if $D\in {\rm Rd}_n[0]\cup\cdots\cup {\rm Rd}_n[r-2]$, then $p_1D$ have at most $r-1$ isolated
vertices in each row and this proves our claim.

If $D\in {\rm Rd}_n[r-1]$, then (using the concatenation rule of rook diagrams) it is clear that
$p_1p_2\cdots p_r$ is involved in $p_1D$, i.e., $p_1D=p_1p_2\cdots p_r$, if and only if there exists
an integer $1\leqslant i\leqslant r$ such that \begin{enumerate}
\item[(1)] $\{1,i^{-}\}$ is an edge of $D$; and
\item[(2)] $\{j,j^{-}\}$ is an edge of $D$ for all $r+1\leqslant j\leqslant n$; and
\item[(3)] all other vertices of $D$ are isolated vertices.
\end{enumerate}
In this case, the number of such rook $n$-diagram $D$ is $r$.

If $D\in {\rm Rd}_n[r]$, it is easy to see that $p_1p_2\cdots p_r$ is involved in $p_1D$,
i.e., $p_1D=p_1p_2\cdots p_r$, if and only if $D=p_1p_2\cdots p_r$.

Therefore, the coefficient of $p_1p_2\cdots p_r$ in the right-hand side of the equality $(3.1)$ is
$rc_{r-1}+c_r$, which implies that $c_r=(-1)rc_{r-1}=(-1)^rr!$ by a simple induction argument.
\end{proof}

\begin{lemma}\label{xx3.8}
The element $X_2$ can be taken of the following form:
$$
X_2=\sum_{\sigma\in\mathfrak{S}_n}\sgn(\sigma)\sigma +\sum_{D\in {\rm Rd}_n[1]}\sgn(D) D.
$$
\end{lemma}

\begin{proof}
Our proof is similar with that given in the above lemma, but we add it here for completeness.
By Lemma \ref{xx3.5}, we can take $X_2$ of the following form:
$$
X_2=\sum_{\sigma\in\mathfrak{S}_n}\sgn(\sigma)\sigma +\sum_{r=1}^n c'_r\sum_{D\in {\rm Rd}_n[r]}\sgn(D) D.
$$
where $c'_r\in F$ for each $1\leqslant r\leqslant n$. It remains to compute these $c'_r$ explicitly.

We first compute $c'_1$. The strategy we shall use is to compare the coefficients of $p_1$ in both
sides of the equality
$$
0=p_1X_2=\sum_{\sigma\in\mathfrak{S}_n}\sgn(\sigma)p_1\sigma +\sum_{r=1}^n c'_r\sum_{D\in {\rm Rd}_n[r]}\sgn(D) p_1D. \eqno(3.2)
$$
Note that by the concatenation rule of rook diagrams $p_1D\in {\rm Rd}_n$ for all $D\in {\rm Rd}_n$
and $p_1$ is involved in $p_1D$ only if $D\in {\rm Rd}_n[0] \cup {\rm Rd}_n[1]$.
If $D\in {\rm Rd}_n[0]$, then it is easy to see that $p_1$ is involved in $p_1D$ ($p_1=p_1D$ in this case)
if and only if $D=1$, the identity element of rook monoid $R_n$. If $D\in {\rm Rd}_n[1]$, then
$p_1$ is involved in $p_1D$ (also $p_1=p_1D$ in this case) if and only if $D=p_1$. Therefore, the
coefficient of $p_1$ in the right-hand side of the equality $(3.2)$ is $1+(-1)c'_1$,
since $\sgn(p_1)=-1$, which implies that $c'_1=1$.

In general, suppose that $2\leqslant r\leqslant n$. The strategy we shall use to compute $c'_r$ is to
compare the coefficients of $p_1p_2\cdots p_r$ in both sides of the equality $(3.2)$. We have known that
$p_1p_2\cdots p_r$ is involved in $p_1D$ only if $D\in {\rm Rd}_n[r-1] \cup {\rm Rd}_n[r]$.

If $D\in {\rm Rd}_n[r-1]$, then (using the concatenation rule of rook diagrams) it is clear that
$p_1p_2\cdots p_r$ is involved in $p_1D$, i.e., $p_1D=p_1p_2\cdots p_r$, if and only if there exists
an integer $1\leqslant i\leqslant r$ such that \begin{enumerate}
\item[(1)] $\{1,i^{-}\}$ is an edge of $D$; and
\item[(2)] $\{j,j^{-}\}$ is an edge of $D$ for all $r+1\leqslant j\leqslant n$; and
\item[(3)] all other vertices of $D$ are isolated vertices.
\end{enumerate}
Let $D$ be such a rook diagram. Then $\ell(D)=(r-1)+(r-i)$ and hence $\sgn(D)=(-1)^{r-i}$.
At the same time, the number of such rook $n$-diagram is $r$.

If $D\in {\rm Rd}_n[r]$, it is easy to see that $p_1p_2\cdots p_r$ is involved in $p_1D$,
i.e., $p_1D=p_1p_2\cdots p_r$, if and only if $D=p_1p_2\cdots p_r$. In this case, $\sgn(D)=(-1)^r$.

Therefore, the coefficient of $p_1p_2\cdots p_r$ in the right-hand side of the equality $(3.2)$ is
$$
c'_{r-1}\sum_{i=1}^r (-1)^{r-i}+c'_r(-1)^r,
$$
which implies that $c'_2=\cdots=c'_n=0$ by a simple induction argument.
\end{proof}

\noindent{\bf Proof of Proposition \ref{xx3.6}.} This follows immediately from Lemma \ref{xx3.7}
and Lemma \ref{xx3.8}. \qed

For convenience, we call $X_1$, $X_2$ the symmetrizer and anti-symmetrizer respectively of
rook monoid $R_n$.
Now we turn to study the decomposition of the rook monoid algebra $FR_n$. Let $S$ be a subset
of $\{1,2,\ldots,n\}$. We define
$$
{\rm Rd}_S:=\{D\in {\rm Rd}_n\ |\ \{j,j^{-}\}\ \text{is an edge of}\ D\
\text{for}\ j\not\in S \}.
$$
It is easy to verify that the set ${\rm Rd}_S$ forms a submonoid of $R_n$ and we denote this
submonoid as $R_S$, which is isomorphic to rook monoid $R_{|S|}$, where $|S|$ means the coordinate
of $S$. Especially, we consider $R_i$ as the submonoid $R_{\{1,2,\ldots,i\}}$ with $1\leqslant i
\leqslant n$. Furthermore, we write $X_{S}$ (resp. $Y_S$) the symmetrizer (resp. anti-symmetrizer)
of the rook monoid $R_S$. If $s_{ij}$ is the simple transposition which interchanges $i$ and $j$
with $i,j\in S$, then $s_{ij}X_S=X_S s_{ij}=X_S$, $s_{ij}Y_S=Y_S s_{ij}=-Y_S$ by definition.
If $i\in S$, then $p_iX_S=X_S p_i=0$, $p_iY_S=Y_S p_i=0$ by definition.

Let $\lambda\vdash r$ be a partition with $0\leqslant r\leqslant n$.
For a given $\lambda_r^n$-tableau $\mathfrak{t}$, recall that $\mathcal{R}_i$ (resp. $\mathcal{C}_j$) is the
set of entries in the $i$-th row (resp. the $j$-th column) of $\mathfrak{t}$.

\begin{definition}\label{xx3.9}
Let $\lambda\vdash r$ with $0\leqslant r\leqslant n$ and $\mathfrak{t}$ a $\lambda_r^n$-tableau. Define
$$
e(\mathfrak{t}):=Y_{\mathcal{C}_1}Y_{\mathcal{C}_2}\cdots Y_{\mathcal{C}_{\lambda_1}}
X_{\mathcal{R}_1}X_{\mathcal{R}_2}\cdots X_{\mathcal{R}_{\ell(\lambda)}}
\prod_{i\not\in {\rm cont}(\mathfrak{t})}p_i.
$$
\end{definition}

It is helpful to point out that $e(\mathfrak{t})\neq 0$. In fact, let
$\phi: FR_n\twoheadrightarrow FR_n/ R_n^{(n-r+1)}$ be the canonical epimorphism. Then the image
$\phi(e(\mathfrak{t}))\neq 0$ by the representation theory of symmetric group
(see \cite[\S 14.7]{van}). Note that when $r=0$, $e(\emptyset)=p_1p_2\cdots p_n$,
where $\emptyset$ denotes the empty partition. For a partition $\lambda\vdash r$ with $0\leqslant r\leqslant n$,
recall that $R^{\lambda}$ is the Specht module corresponding to $\lambda$ (see the Section \ref{xxsec2.2}).

\begin{lemma}\label{xx3.10}
Let $\lambda\vdash r$ be a partition with $0\leqslant r\leqslant n$ and $\mathfrak{t}$ a $\lambda_r^n$-tableau.
For any partition $\mu\vdash f$ with $0\leqslant f\leqslant n$, we have
$e(\mathfrak{t})R^{\mu}=0$ unless $\lambda=\mu$.
\end{lemma}

\begin{proof}
Let $a$ be an arbitrary nonzero element of $R^{\mu}$. We denote $T(\mu)$ the set of all $\mu_f^n$-tableau
for convenience. Write $a$ as a linear combination of $n$-polytabloids:
$$
a=\sum_{\mathfrak{s}\in T(\mu)}a_{\mathfrak{s}}e_{\mathfrak{s}},
$$
where $a_{\mathfrak{s}}\in F$. We now compute $e(\mathfrak{t})a$ and there are three cases occurring.

\textbf{Case 1.} $|\mu|>|\lambda|$. In this case, for an arbitrary $\mu_f^n$-tableau $\mathfrak{s}$,
there exists an integer $i\in {\rm cont}(\mathfrak{s})$ but $i\not\in {\rm cont}(\mathfrak{t})$.
Hence $p_i\mathfrak{s}=0$ (see the paragraph below Definition \ref{xx2.3}). It follows from the
Definition \ref{xx3.9} and Lemma \ref{xx2.4} that $e(\mathfrak{t})e_{\mathfrak{s}}=
e(\mathfrak{t})p_ie_{\mathfrak{s}}=0$ and hence $e(\mathfrak{t})a=0$.

\textbf{Case 2.} $|\mu|<|\lambda|$. From the proof of Case 1, we have
$$
e(\mathfrak{t})a=e(\mathfrak{t})\sum_{\mathfrak{s}\in T(\mu),{\rm cont}(\mathfrak{s})
\subseteq {\rm cont}(\mathfrak{t})}a_{\mathfrak{s}}e_{\mathfrak{s}}.
$$
Since $|\mu|<|\lambda|$, for an arbitrary $\mu_f^n$-tableau $\mathfrak{s}$, there
exists an integer $j\in {\rm cont}(\mathfrak{t})$ but $j\not\in {\rm cont}(\mathfrak{s})$.
Hence $p_j\mathfrak{s}=\mathfrak{s}$ and thus $p_j\{\mathfrak{s}\}=\{\mathfrak{s}\}$ by the action
of $R_n$ on $M^{\mu}$. Note that each term in the linear combination of tabloids that form
$e_{\mathfrak{s}}$ contains the exact same entries as $\mathfrak{s}$. Therefore,
$$
p_je_{\mathfrak{s}}=\sum_{\sigma\in C_{\mathfrak{s}}} \sgn(\sigma)p_j (\sigma\{\mathfrak{s}\})=
\sum_{\sigma\in C_{\mathfrak{s}}} \sgn(\sigma)p_j \{\sigma\mathfrak{s}\}=e_{\mathfrak{s}}.
$$
On the other hand, the elements $X_{\mathcal{R}_1}, X_{\mathcal{R}_2},\ldots, X_{\mathcal{R}_{\ell(\lambda)}}$
pairwise commute with each other. Since $j\in{\rm cont}(\mathfrak{t})$, there is
$X_{\mathcal{R}_1}X_{\mathcal{R}_2}\cdots X_{\mathcal{R}_{\ell(\lambda)}}p_j=0$.
Hence for an arbitrary $\mu_f^n$-tableau $\mathfrak{s}$, we have
$$
e(\mathfrak{t})e_{\mathfrak{s}}=e(\mathfrak{t})(p_je_{\mathfrak{s}})=
Y_{\mathcal{C}_1}Y_{\mathcal{C}_2}\cdots Y_{\mathcal{C}_{\lambda_1}}
(X_{\mathcal{R}_1}X_{\mathcal{R}_2}\cdots X_{\mathcal{R}_{\ell(\lambda)}}p_j)
\prod_{i\not\in {\rm cont}(\mathfrak{t})}p_i e_{\mathfrak{s}}=0.
$$
Then $e(\mathfrak{t})a=0$ when $|\mu|<|\lambda|$.

\textbf{Case 3.} $|\mu|=|\lambda|$. From the proof of Case 1, we have
$$\begin{aligned}
e(\mathfrak{t})a&=e(\mathfrak{t})\sum_{\mathfrak{s}\in T(\mu),{\rm cont}(\mathfrak{s})
\subseteq {\rm cont}(\mathfrak{t})}a_{\mathfrak{s}}e_{\mathfrak{s}}\\
&=e(\mathfrak{t})\sum_{\mathfrak{s}\in T(\mu),{\rm cont}(\mathfrak{s})
= {\rm cont}(\mathfrak{t})}a_{\mathfrak{s}}e_{\mathfrak{s}}.
\end{aligned}$$
Recall that $C_{\mathfrak{t}}$ is the column stabiliser subgroup of $\mathfrak{t}$ in
symmetric group $\mathfrak{S}_n$. Let $R_{\mathfrak{t}}:=\mathfrak{S}_{\mathcal{R}_1}\times \mathfrak{S}_{\mathcal{R}_2}\times
\cdots\times \mathfrak{S}_{\mathcal{R}_{\ell(\lambda)}}$ be the row stabiliser subgroup of $\mathfrak{t}$ in
symmetric group $\mathfrak{S}_n$. At the same time, for any $\mathfrak{s}\in T(\mu)$ satisfying
${\rm cont}(\mathfrak{s})={\rm cont}(\mathfrak{t})$, we have $p_je_{\mathfrak{s}}=0$ for all
$j\in {\rm cont}(\mathfrak{t})$ by Lemma \ref{xx2.4}. Hence it follows from the definition of $e(\mathfrak{t})$ that
$$
e(\mathfrak{t})a=\sum_{q\in C_{\mathfrak{t}}}\sum_{p\in R_{\mathfrak{t}}}\sgn(q)qp
\sum_{\mathfrak{s}\in T(\mu),{\rm cont}(\mathfrak{s})
= {\rm cont}(\mathfrak{t})}a_{\mathfrak{s}}e_{\mathfrak{s}}.
$$
Now by the representation theory of symmetric group (see \cite[\S 14.7]{van}), we have
$e(\mathfrak{t})a=0$ unless $\lambda=\mu$. This completes the proof of the lemma.
\end{proof}

Let $\lambda\vdash r$ be a partition with $0\leqslant r\leqslant n$. Let $\mathfrak{t}^{\lambda}$ be
the $\lambda_r^n$-tableau in which the numbers $1,2,\ldots,r$ appear in order along successive rows.

\begin{definition}\label{xx3.11}
Let $\lambda\vdash r$ with $0\leqslant r\leqslant n$. Define a two-sided ideal
$$
I(\lambda):=FR_ne(\mathfrak{t}^{\lambda})FR_n.
$$
\end{definition}

We now at the position to give out the decomposition of rook monoid algebra $FR_n$ as the direct sum of blocks.

\begin{theorem}\label{xx3.12}
For each partition $\lambda\vdash r$ with $0\leqslant r\leqslant n$, $I(\lambda)$ is a
minimal two-sided ideal and the rook monoid algebra can be decomposed as
$$
FR_n=\bigoplus_{r=0}^n\bigoplus_{\lambda\vdash r}I(\lambda).
$$
\end{theorem}

\begin{proof}
Munn proved that the rook monoid algebra $FR_n$ is semisimple in \cite[Theorem 3.1]{Munn2}. By the well-known
Wedderburn-Artin Theorem and Theorem \ref{xx2.5}, there is
$$
FR_n=\bigoplus_{r=0}^n\bigoplus_{\lambda\vdash r}I_{\lambda}
\cong\bigoplus_{r=0}^n\bigoplus_{\lambda\vdash r}M_{n_{\lambda}}(D_{\lambda}),\eqno(3.3)
$$
where $I_{\lambda}$ is a simple subalgebra which is isomorphic to
the full matrix algebra $M_{n_{\lambda}}(D_{\lambda})$ of degree $n_{\lambda}$ and $D_{\lambda}$ is a finite
dimensional division algebra over $F$. Furthermore, $M_{n_{\lambda}}(D_{\lambda})$ corresponds to
the irreducible Specht module $R^{\lambda}$, i.e., $R^{\lambda}$ is the unique (up to isomorphism)
irreducible module of simple algebra $I_{\lambda}$.
On the other hand, East \cite{East} proved that the
rook monoid algebra $FR_n$ is a cellular algebra in the sense of Graham and Lehrer \cite{GL}.
By the general theory of cellular algebra \cite[Proposition 2.6 and Theorem 3.4]{GL},
$\End_{FR_n}(R^{\lambda})\cong F$ for all $\lambda\vdash r$ with $0\leqslant r\leqslant n$, i.e.,
each Specht module $R^{\lambda}$ is absolutely irreducible. That means the field
$F$ is a splitting field for $FR_n$ and hence $D_{\lambda}=F$.

Let $\lambda\vdash r$ be a partition with $0\leqslant r\leqslant n$ and $\mathfrak{t},\mathfrak{s}\in T(\lambda)$.
We have from the equation $(3.3)$ that $I_{\lambda}R^{\mu}=0$ unless $\lambda=\mu$
and $I_{\lambda}R^{\lambda}=R^{\lambda}$. Note that the nonzero ideal
$\langle e(\mathfrak{t})\rangle$ is a sum of certain ideals $I_{\mu}$.
Then it follows from Lemma \ref{xx3.10} that the nonzero two-sided ideal $\langle e(\mathfrak{t})\rangle
=I_{\lambda}$ and hence $\langle e(\mathfrak{t})\rangle=\langle e(\mathfrak{s})\rangle$ is minimal.
Especially, $I(\lambda)=\langle e(\mathfrak{t^{\lambda}})\rangle=I_{\lambda}$
and the theorem is proved.
\end{proof}

\section{Proof of Theorem \ref{xx1.2}}\label{xxsec4}

In this section we shall give the main result of this paper. That is, the proof of Theorem \ref{xx1.2}.

For any positive integer $k\leqslant n$, the natural map$
s_i\mapsto s_i, p_j\mapsto p_j$
for all $1\leqslant i\leqslant k-1$ and $1\leqslant j\leqslant k$ extends to an algebra embedding from
$FR_k$ into $FR_n$, i.e., $R_k$ considered as the submonoid $R_{\{1,2,\ldots,k\}}$.
From this point of view, when $m<n$ (where $m=\dim(V)$), the anti-symmetrizer $Y_{m+1}:=
Y_{\{1,2,\ldots,m+1\}}$ of $FR_{m+1}$ is an element of $FR_n$. That is
$$
Y_{m+1}=\sum_{\sigma\in\mathfrak{S}_{m+1}}\sgn(\sigma)\sigma +\sum_{D\in {\rm Rd}_{m+1}[1]}\sgn(D) D\in FR_n.
$$
By Theorem \ref{xx3.12}, the two-sided ideal $\Ann_{FR_n}\bigl( U^{\otimes n}\bigr)
=\Ker(\varphi)$ is a sum of certain ideals $I(\lambda)$. We define a two-sided ideal
$$
I_{m+1}=\sum_{r=0}^n \sum_{\substack{\lambda\vdash r\\
\ell(\lambda)\geqslant m+1}}I(\lambda).
$$
We shall prove Theorem \ref{xx1.2} in three parts by showing that
$$
\langle Y_{m+1}\rangle \subseteq \Ker(\varphi)\subseteq I_{m+1}\subseteq\langle Y_{m+1}\rangle.
$$

For convenience, we set $I(m,n):=\{(i_1,\ldots,i_n)\ |\ i_j\in\{0,1,\ldots,m\}, \forall j\}$.
For any $\underline{i}=(i_1,\ldots,i_n)\in I(m,n)$, we write $v_{\underline{i}}=v_{i_1}\otimes
\cdots\otimes v_{i_n}$ for a simple tensor. Let's start the proof by a technical lemma.

\begin{lemma}\label{xx4.1}
Let $D=d_1^{-1}p_1p_2\cdots p_r\sigma d_2$ be a rook $n$-diagram defined in Proposition \ref{xx2.1}.
For any simple tensor $v_{\underline{i}}\in U^{\otimes n}$, if
$Dv_{\underline{i}}\neq 0$, then $Dv_{\underline{i}}=d_1^{-1}\sigma d_2v_{\underline{i}}$.
\end{lemma}

\begin{proof}
By Proposition \ref{xx2.1}, the isolated vertices in the bottom row of $D$ are labeled by
$((1)d_2)^-$, $((2)d_2)^-$, $\ldots, ((r)d_2)^-$. Hence if $Dv_{\underline{i}}\neq 0$, there is
$i_{(1)d_2}=i_{(2)d_2}=\cdots=i_{(r)d_2}=0$. In this case, we have
$$\begin{aligned}
Dv_{i_1}\otimes v_{i_2}\otimes\cdots\otimes v_{i_n}&=(d_1^{-1}p_1p_2\cdots p_r\sigma d_2)v_{i_1}\otimes v_{i_2}
\otimes\cdots\otimes v_{i_n}\\
&=d_1^{-1}\sigma p_1p_2\cdots p_r v_{i_{(1)d_2}}\otimes\cdots\otimes v_{i_{(r)d_2}}\otimes
v_{i_{(r+1)d_2}}\otimes\cdots\otimes v_{i_{(n)d_2}}\\
&=d_1^{-1}\sigma v_{i_{(1)d_2}}\otimes\cdots\otimes v_{i_{(r)d_2}}\otimes
v_{i_{(r+1)d_2}}\otimes\cdots\otimes v_{i_{(n)d_2}}\\
&=d_1^{-1}\sigma d_2v_{i_1}\otimes v_{i_2}\otimes\cdots\otimes v_{i_n}.
\end{aligned}$$
This completes the proof of the lemma.
\end{proof}

\begin{lemma}\label{xx4.2}
With notations as above, there is $\langle Y_{m+1}\rangle \subseteq \Ker(\varphi)$.
\end{lemma}

\begin{proof}
For any simple tensor $v_{\underline{i}}\in U^{\otimes n}$, we only need to proof $Y_{m+1}v_{\underline{i}}=0$.
By the actions of rook monoids on $n$-tensor spaces
defined in Section \ref{xxsec2.1}, we know that $Y_{m+1}$ only acts on the first $m+1$ components
of $v_{\underline{i}}$. Hence we can assume $n=m+1$ without lose of the generality. For an arbitrary
simple tensor $v_{\underline{i}}=v_{i_1}\otimes v_{i_2}\otimes\cdots\otimes v_{i_{m+1}}$, if the
$(m+1)$-tuple $(i_1,i_2,\ldots,i_{m+1})$ has a repeated number, for instance, $i_j=i_k$ with $j<k$,
then obviously $Y_{m+1}v_{\underline{i}}=Y_{m+1}s_{jk}v_{\underline{i}}=-Y_{m+1}v_{\underline{i}}$
and hence $Y_{m+1}v_{\underline{i}}=0$, where $s_{jk}$ is the transposition which interchanges $j$
and $k$.

Then, we assume that $i_1,i_2,\ldots,i_{m+1}$ are different with each other. Noting that $\dim(V)=m$,
we can assume $v_{\underline{i}}=v_0\otimes v_1\otimes\cdots\otimes v_m$ without lose of the generality.
Therefore, for each $D\in {\rm Rd}_{m+1}[1]$, $Dv_{\underline{i}}\neq 0$ implies that the first vertex
in the bottom row of $D$ is isolated. In other words, $D=d_1^{-1}p_1\sigma$ with $d_1\in\mathcal{D}_1$,
$\sigma\in\mathfrak{S}_{\{2,3,\ldots,m+1\}}$. Then we have from Lemma \ref{xx4.1} that
$$\begin{aligned}
Y_{m+1}v_{\underline{i}}&=\sum_{\sigma\in\mathfrak{S}_{m+1}}\sgn(\sigma)\sigma v_{\underline{i}}
+\sum_{D\in {\rm Rd}_{m+1}[1]}\sgn(D) D v_{\underline{i}}\\
&=\sum_{\sigma\in\mathfrak{S}_{m+1}}\sgn(\sigma)\sigma v_{\underline{i}}
+\sum_{\substack{D=d^{-1}p_1\sigma\\ d\in\mathcal{D}_1, \sigma\in\mathfrak{S}_{\{2,\ldots,m+1\}}}}\sgn(D)D v_{\underline{i}}\\
&=\sum_{\sigma\in\mathfrak{S}_{m+1}}\sgn(\sigma)\sigma v_{\underline{i}}
-\sum_{\substack{d^{-1}\sigma\\ d\in\mathcal{D}_1, \sigma\in\mathfrak{S}_{\{2,\ldots,m+1\}}}}\sgn(d^{-1}\sigma)d^{-1}\sigma v_{\underline{i}}\\
&=\left(\sum_{\sigma\in\mathfrak{S}_{m+1}}\sgn(\sigma)\sigma-
\sum_{\substack{d^{-1}\sigma\\ d\in\mathcal{D}_1, \sigma\in\mathfrak{S}_{\{2,\ldots,m+1\}}}}\sgn(d^{-1}\sigma)d^{-1}\sigma \right)v_{\underline{i}}\\
&=0,
\end{aligned}$$
where the last identity follows from the fact that $\{d^{-1}|d\in\mathcal{D}_1\}$ is a set of
left coset representatives of $\mathfrak{S}_{\{1\}}\times \mathfrak{S}_{\{2,\ldots,m+1\}}$ in $\mathfrak{S}_{m+1}$.
\end{proof}

\begin{lemma}\label{xx4.3}
With notations as above, there is $\Ker(\varphi)\subseteq I_{m+1}$.
\end{lemma}

\begin{proof}
As mentioned at the beginning of this section, $\Ker(\varphi)$ is a sum of certain ideals $I(\lambda)$.
We equivalently show that if $I(\lambda)\nsubseteq I_{m+1}$, then $I(\lambda)\nsubseteq \Ker(\varphi)$.
Taking a partition $\lambda$ such that $I(\lambda)\nsubseteq I_{m+1}$, there exists an integer
$0\leqslant r\leqslant n$ such that $\lambda\vdash r$ and $\ell(\lambda)\leqslant m$.
We shall prove $I(\lambda)\nsubseteq \Ker(\varphi)$ by finding a simple tensor $v_{\underline{i}}\in U^{\otimes n}$
such that $e(\mathfrak{t}^{\lambda})v_{\underline{i}}\neq 0$.

Let$$
v_{\underline{i}}=\underbrace{v_1\otimes\cdots\otimes v_1}_{\text{$\lambda_1$ copies}}\otimes
\underbrace{v_2\otimes\cdots\otimes v_2}_{\text{$\lambda_2$
copies}}\otimes\cdots\otimes \underbrace{v_m\otimes\cdots\otimes
v_m}_{\text{$\lambda_m$ copies}}\otimes\underbrace{v_0\otimes\cdots\otimes v_0}_{\text{$n-r$ copies}}.
$$
For each rook $n$-diagram $D\in R_n^{(n-r+1)}$, it is clear that $Dv_{\underline{i}}=0$. Therefore,
$$\begin{aligned}
e(\mathfrak{t}^{\lambda})v_{\underline{i}}&=\sum_{q\in C_{\mathfrak{t^{\lambda}}}}
\sum_{p\in R_{\mathfrak{t^{\lambda}}}}\sgn(q)qp v_{\underline{i}}\\
&=|R_{\mathfrak{t^{\lambda}}}|\sum_{q\in C_{\mathfrak{t^{\lambda}}}}\sgn(q)q v_{\underline{i}},
\end{aligned}$$
where $|R_{\mathfrak{t^{\lambda}}}|=\lambda_1!\lambda_2!\cdots \lambda_m!$ is the order of the group
$R_{\mathfrak{t^{\lambda}}}$. For each $q\in C_{\mathfrak{t^{\lambda}}}$, if $q\neq 1$, then
$qv_{\underline{i}}\neq v_{\underline{i}}$, since each element of $C_{\mathfrak{t^{\lambda}}}$ except the identity
takes at least one entry of some column of $\mathfrak{t^{\lambda}}$ to a different row.
Hence the coefficient of $v_{\underline{i}}$ in $e(\mathfrak{t}^{\lambda})v_{\underline{i}}$
is $|R_{\mathfrak{t^{\lambda}}}|\neq 0$, and this completes the proof of this lemma.
\end{proof}

Let $\lambda\vdash r$ be a partition with $0\leqslant r\leqslant n$. Let $\mathfrak{t}_{\lambda}$ be
the $\lambda_r^n$-tableau in which the numbers $1,2,\ldots,r$ appear in order along successive columns.

\begin{lemma}\label{xx4.4}
With notations as above, there is $I_{m+1}\subseteq\langle Y_{m+1}\rangle$.
\end{lemma}

\begin{proof}
Let $\lambda\vdash r$ with $0\leqslant r\leqslant n$ and $\ell(\lambda)\geqslant m+1$. We need to prove
$I(\lambda)\subseteq \langle Y_{m+1}\rangle$. It follows from the proof of Theorem \ref{xx3.12} that
$I(\lambda)=\langle e(\mathfrak{t}_{\lambda})\rangle$ and hence we only need to show $e(\mathfrak{t}_{\lambda})
\in \langle Y_{m+1}\rangle$.

Since $\ell(\lambda)\geqslant m+1$, the set $\mathcal{C}_1$ of entries in the first column of $\mathfrak{t}_{\lambda}$
is $\{1,2,\ldots,l\}$, where $l=\ell(\lambda)$.
By a direct computation, we have
$$\begin{aligned}
Y_{m+1}e(\mathfrak{t}_{\lambda})&=Y_{m+1}Y_{\mathcal{C}_1}Y_{\mathcal{C}_2}\cdots Y_{\mathcal{C}_{\lambda_1}}
X_{\mathcal{R}_1}X_{\mathcal{R}_2}\cdots X_{\mathcal{R}_{\ell(\lambda)}}
\prod_{i\not\in {\rm cont}(\mathfrak{t})}p_i\\
&=Y_{m+1}Y_{l}Y_{\mathcal{C}_2}\cdots Y_{\mathcal{C}_{\lambda_1}}
X_{\mathcal{R}_1}X_{\mathcal{R}_2}\cdots X_{\mathcal{R}_{\ell(\lambda)}}
\prod_{i\not\in {\rm cont}(\mathfrak{t})}p_i\\
&=(m+1)!Y_{l}Y_{\mathcal{C}_2}\cdots Y_{\mathcal{C}_{\lambda_1}}
X_{\mathcal{R}_1}X_{\mathcal{R}_2}\cdots X_{\mathcal{R}_{\ell(\lambda)}}
\prod_{i\not\in {\rm cont}(\mathfrak{t})}p_i\\
&=(m+1)!e(\mathfrak{t}_{\lambda}).
\end{aligned}$$
Therefore $e(\mathfrak{t}_{\lambda})\in \langle Y_{m+1}\rangle$ and this completes the proof of the lemma.
\end{proof}

\noindent{\bf Proof of Theorem \ref{xx1.2}.} It follows immediately from Lemmas \ref{xx4.2}, \ref{xx4.3}
and \ref{xx4.4}. \qed

\end{document}